\documentclass[11pt]{amsart}
\usepackage{amsmath,amsfonts,amssymb,amscd,amsthm,amsbsy,upref,epsf}
\newtheorem{theo}{Theorem}
\newtheorem{prop}[theo]{Proposition}

\newtheorem{lemm}[theo]{Lemma}

\newcommand{\R}{\mathbb R}

\newcommand{\dd} {\mathrm{d}}
\newcommand{\La}{{\Lambda}}

\newcommand{\eps}{{\varepsilon}}

\def\osc{\operatorname{osc}}

\newcommand{\bu}{\bar{u}}

\newcommand{\dt}{{\partial_t}}

\title[H\"older regularity for Hamilton-Jacobi equations]{De Giorgi techniques applied to the H\"older regularity of solutions to Hamilton-Jacobi equations}
\author[Chan]{Chi Hin Chan}
\address[Chi Hin Chan]{\newline Department of Applied Mathematics, \newline National Chiao Tung University,1001 Ta Hsueh Road, Hsinchu, Taiwan 30010, ROC}
\email{cchan@math.nctu.edu.tw}

\author[Vasseur]{Alexis F. Vasseur}
\address[Alexis F. Vasseur]{\newline Department of Mathematics, \newline The University of Texas at Austin, Austin, TX 78712, USA}
\email{vasseur@math.utexas.edu}

\date{\today}


\thanks{\textbf{Acknowledgment.} A. F. Vasseur was partially supported by the NSF Grant DMS 1209420. Ch.-H. Chan was partially supported by the grant NSC 103-2115-M-009-010-MY2 from the National Science Council of Taiwan.
}
\begin{document}

\begin{abstract}
This article is dedicated to the proof of $C^\alpha$ regularization effects of Hamilton-Jacobi equations. The proof is based on the De Giorgi method. The regularization is independent on the regularity of the Hamiltonian.
\end{abstract}

\maketitle \centerline{\date}

\section{Introduction}
This article is dedicated to the proof of $C^\alpha$ regularization effects of  Hamilton-Jacobi equations of the form:
\begin{equation}\label{eq:main}
\dt u+H(t,x, \nabla u)=0,\qquad t\in (0,T),\ \ x\in \R^N,
\end{equation}
with $T>0$, $N\in \mathbb{Z}^+$, and where the Hamiltonian verifies a uniform, in $x$ and $t$, coercivity property of the form:
\begin{equation}\label{eq:coercivity}
\frac{1}{\La} |P|^p-\Lambda\leq H(t,x,P)\leq \La |P|^p+\La,
\end{equation}
for a $p \in (1, N)$ and $\Lambda \geq 1$.
The main theorem, which is the focus of this article, is the following.
\begin{theo}\label{theo:main}
Let $N \in \mathbb{Z}^+$ and $p \in (1,N)$ to be given, and let $u$ be a a bounded viscosity solution on $(0,T)\times\R^N$ to \eqref{eq:main}, with a Hamiltonian $H(t,x , P)$ satisfying coercivity property \eqref{eq:coercivity}. Then, it follows that, for each $\delta \in (0,T)$, we have $u\in C^\alpha([\delta,T)\times\R^N)$, where $\alpha \in (0,1)$, and $\|u\|_{C^\alpha([\delta, T)\times \mathbb{R}^N)}$ depend only on $N$, $\delta$, $\Lambda$, $p$ and $\|u\|_{L^\infty ((0,T)\times \mathbb{R}^N)}$.
\end{theo}
The result of the paper is not new. However the method of proof, based on the De Giorgi method \cite {DeGiorgi} to study the regularity of elliptic equations with rough coefficients, is pretty unusual for the study of viscosity solutions, and should lead to new results in this area, as the study of regularity of solutions to  nonlocal "Hamilton-Jacobi like" equations (\cite{CVencours}).  Our proof is inspired by  previous applications of the De Giorgi method to integral-differential parabolic equations \cite{CCV,CaffarelliVasseur2010}.
\vskip0.3cm
The first H\"older regularity result of this kind was obtained by Schwab, in \cite{Schw-09}, in the case of a convex Hamiltonian. The result was a key ingredient to perform the  stochastic homogenization of Hamilton-Jacobi equations in Stationary Ergodic Spatio-Temporal media. This result inspired several generalizations. The non-convex case is technically more challenging. The first proofs, relying  on stochastic methods, were obtained by Cardaliaguet \cite{Card2009}, and Cannarsa Cardaliaguet \cite{CannCard2010}.
 Cardaliaguet and Silvestre provided a simpler proof in \cite{CardSilvestre}. Their proof is based on the construction of sub-solutions and supersolutions combined with improvement of oscillation techniques. It includes applications to some degenerated parabolic equations (and also includes our case).
\vskip0.3cm
Hamilton-Jacobi equations have solutions with breakdown of the $C^1$-regularity in finite time, due to the formation of so-called caustics.
It is quite remarkable that a typical Hamilton-Jacobi equation has some regularization effect on its solutions at a lower level $C^\alpha$, for some $\alpha \in (0,1)$.
This effect was first observed for time-independent, degenerated, elliptic Hamilton-Jacobi equations by Barles in \cite{Barles}.
\vskip0.3cm
Our proof uses the coercivity of the Hamilton-Jacobi equation to induce a parabolic-like regularization effect. It is based on De Giorgi techniques which provide $C^\alpha$-regularization for elliptic equations with rough coefficients. It involves the decrease of the oscillation
of the solution from scale to scale. While obtaining improved oscillation of the solution from above, we only use that the solution verifies (\ref{eq:main}) in the sense of distributions (the ``viscosity solution" structure, based on the comparison principle, is not used).  When we need to shrink the oscillation of the solution by below, the regularization effect is obtained backward in time. In the backward in time regularization process, the ``viscosity structure" of the solution is still irrelevant. However, to be consistent with the regularization by above, the backward in time regularization needs to be pushed back in positive time. This is the only part of the proof which needs the comparison principle.
\vskip0.3cm
Note that  $u+\Lambda t$  verifies
\begin{equation}\label{superviscosity}
\dt v +\La |\nabla v|^p\geq0.
\end{equation}
This inequality is slightly better in the rescaling process.
So, without loss of generality, we will assume that
\begin{equation}\label{eq:coercivitybis}
\frac{1}{\La} |P|^p-\Lambda\leq H(P)\leq \La |P|^p,
\end{equation}
instead of (\ref{eq:coercivity}).

\vskip0.1cm
\noindent{\bf Remark:} It would help a lot to have, at every scale,
$$
\dt v +\La |P|^p\geq \Lambda.
$$
But this inequality will not be preserved via the scaling.
\vskip0.3cm
The rest of the paper is structured as follows. In Section \ref{section:first}, we derive the first lemma of De Giorgi.  In Section \ref{section:second}, we prove the second lemma of De Giorgi. In Section \ref{section:oscillation} we show how the oscillation can be reduced locally. The precise scaling leading to the $C^\alpha$ regularity is provided in Section \ref{section:final}.

\section{The first De-Giorgi's Lemma}\label{section:first}

In this section, we will consider weak (distributional) solutions to the following differential inequality, with $\Lambda \geq 1$ and $1 < p < N$ to be some given constants.
\begin{equation}\label{basicineqONE}
\dt u +\frac{1}{\Lambda}|\nabla u|^p\leq \Lambda .
\end{equation}
We recall that every viscosity solution to (\ref{basicineqONE}) is solution in the sense of distribution of the same equation (see Crandall and Lions \cite{CrandallLions1982}).

\begin{lemm}\label{lemm:first}
There exists an absolute constant $\delta = \delta (N,\Lambda , p)>0$ such that, for any weak solution $u : [0,2]\times B(1) \rightarrow \mathbb{R}$ to \eqref{basicineqONE} on $[0,2]\times B(1)$, we have the following implication:

If it happens that
$$
\int_{ [0,2]\times B(1)}u_+\leq \delta,
$$
then it follows that
$$
u\leq 1\qquad \mathrm{on} \ \ [1,2]\times B(1).
$$
Here, $B(r)$ stands for the open ball centered at the origin $O$ with radius $r$ in $\mathbb{R}^N$.
\end{lemm}

\begin{proof}
Let $u :[0,2]\times B(1) \rightarrow \mathbb{R}$ to be a weak solution to \eqref{basicineqONE} on $[0,2]\times B(1)$ which satisfies all the hypothesis in Theorem \ref{lemm:first}. First, we work with a sequence of truncated functions $v_k$ on $[0,2]\times B(1)$ which is defined as follows.
\begin{equation*}
v_k = \big ( u - (1-\frac{1}{2^k}) \big )_+ .
\end{equation*}
Next, we multiply \eqref{basicineqONE} by $\chi_{\{v_k > 0\}}$ to yield the following relation
\begin{equation}\label{truncatedineqONE}
\dt v_k + \frac{1}{\Lambda} |\nabla v_k|^p \leq \Lambda \chi_{\{v_k > 0\}} ,
\end{equation}
which holds in the distributional sense on $[0,2]\times B(1)$, for each integer $k\geq 1$.\\

Inequality \eqref{truncatedineqONE} is the ground on which we will build up a nonlinear recurrence relation for the following sequence of truncated energies.
\begin{equation*}
U_k = \sup_{t\in [T_k , 2]} \int_{B(1)} v_k(t , \cdot ) + \int_{T_k}^2 \int_{B(1)} |\nabla v_k (t ,y  )|^p \dd y  \dd t ,
\end{equation*}
where $T_k = 1-\frac{1}{2^k}$, for each $k \in \mathbb{Z}^+$.
Next, take any two real numbers $\sigma$, $t$ which satisfy the following constraint.
\begin{equation*}
T_{k-1} \leq \sigma \leq T_k \leq t \leq 2.
\end{equation*}
By taking the spatial-integral over $B(1)$ and then the time-integral over the interval $[\sigma , t]$ for each term in \eqref{truncatedineqONE}, we yield
\begin{equation}\label{good}
\int_{B(1)} v_k(t,\cdot ) + \frac{1}{\Lambda}\int_{\sigma}^t \int_{B(1)} |\nabla v_k|^p(\tau , \cdot ) d\tau
\leq \int_{B(1)} v_k(\sigma ,\cdot ) + \Lambda \int_{\sigma}^t \int_{B(1)} \chi_{\{v_k > 0\}}
\end{equation}
Next, by taking the time average over $\sigma \in [T_{k-1} , T_k]$ for each term in the above inequality, we easily yield
\begin{equation*}
\int_{B(1)} v_k(t,\cdot ) + \frac{1}{\Lambda}\int_{T_k}^t \int_{B(1)} |\nabla v_k|^p(\tau , \cdot ) d\tau
\leq 2^k \int_{T_{k-1}}^{T_k}\int_{B(1)} v_k(\sigma ,\cdot ) + \Lambda \int_{T_{k-1}}^t \int_{B(1)} \chi_{\{v_k > 0\}} ,
\end{equation*}
from which it follows, through taking the $\sup$ over $t\in [T_k , 2]$, that the following relation holds.
\begin{equation*}
\sup_{t\in [T_k , 2]}\int_{B(1)} v_k(t, \cdot ) + \frac{1}{\Lambda}\int_{T_k}^2 \int_{B(1)} |\nabla v_k|^p
\leq \int_{[T_{k-1}, 2]\times B(1)} \big ( 2^k v_k + \Lambda \chi_{\{ v_k >0 \}} \big ) .
\end{equation*}
The above inequality immediately gives
\begin{equation}\label{raiseprep}
U_k \leq \big ( 2^k \Lambda + \Lambda^2  \big ) \bigg \{ \int_{[T_{k-1} , 2]\times B(1)} v_k +
\int_{[T_{k-1} , 2] \times B(1)} \chi_{\{ v_k > 0 \}}  \bigg \}.
\end{equation}
Since $1<p<N$, by using the Sobolev's embedding theorem, we have the following estimate, with $C(N,p) > 0$ to be some absolute constant.
\begin{equation*}
\bigg \| v_k - \frac{1}{|B(1)|} \int_{B(1) } v_k     \bigg \|_{L^{\frac{Np}{N-p}}(B(1))} \leq C(N,p) \big \|\nabla v_k \big \|_{L^p(B(1))} ,
\end{equation*}
from which it follows that
\begin{equation*}
\begin{split}
\big \|v_k \big \|_{L^{\frac{Np}{N-p}}(B(1))} &\leq \frac{|B(1)|^{\frac{1}{p} - \frac{1}{N}}}{|B(1)|}  \int_{B(1)} v_k  +
\bigg \| v_k - \frac{1}{|B(1)|} \int_{B(1) } v_k     \bigg \|_{L^{\frac{Np}{N-p}}(B(1))} \\
& \leq |B(1)|^{\frac{1}{p}-\frac{1}{N}-1} \int_{B(1)} v_k + C(N,p) \big \| \nabla v_k \big \|_{L^p(B(1))} .
\end{split}
\end{equation*}
By raising up the power $p$ on both sides of the above estimate and then taking integration over $t \in [T_k , 2]$, we yield
\begin{equation*}
\begin{split}
\int_{T_k}^2 \big \| v_k (t , \cdot ) \big \|_{L^{\frac{Np}{N-p}}(B(1))}^p \dd t & \leq C(N,p)  \int_{T_k}^2 \bigg \{ \bigg ( \int_{B(1)} v_k(t, \cdot ) \bigg )^p  + \big \|\nabla v_k (t, \cdot ) \big \|_{L^p(B(1))}^p  \bigg \} \dd t \\
& \leq  C(N,p) \bigg \{ 2 U_k^p + \int_{T_k}^2  \big \|\nabla v_k (t , \cdot ) \big \|_{L^p(B(1))}^p  \dd t     \bigg \} \\
& \leq  C(N,p) \bigg \{ 2U_k^p + U_k   \bigg \} .
\end{split}
\end{equation*}
That is, we have
\begin{equation*}
\big \| v_k \big \|_{L^p (T_k , 2 ; L^{\frac{Np}{N-p}}(B(1)) )} \leq C(N,p) \big \{ U_k + U_k^{\frac{1}{p}}  \big \} ,
\end{equation*}
from which it follows, by means of interpolation, that
\begin{equation*}
\begin{split}
\big \|v_k\big \|_{L^{p(1+ \frac{1}{N})} ([T_k , 2]\times B(1))} & \leq \big \| v_k \big \|_{L^{\infty}(T_k , 2 ; L^1(B(1)))}^{1-\frac{N}{N+1}}
\cdot \big \| v_k  \big \|_{L^p ( T_k , 2 ; L^{\frac{Np}{N-p}} (B(1))  )}^{\frac{N}{N+1}} \\
& \leq C(N,p) U_k^{\frac{1}{N+1}} \big \{ U_k + U_k^{\frac{1}{p}} \big \}^{\frac{N}{N+1}} .
\end{split}
\end{equation*}
That is, we have
\begin{equation}\label{Luckyinequality}
\int_{[T_k ,2]\times B(1)} |v_k|^{p(1+\frac{1}{N})} \leq C(N,p) U_k^{\frac{p}{N}} \big \{ U_k + U_k^{\frac{1}{p}}  \big \}^p .
\end{equation}
By means of \eqref{Luckyinequality}, we can now raise up the index for the two terms appearing in \eqref{raiseprep} as follows.
\begin{equation}\label{raiseindexONE}
\begin{split}
\int_{[T_{k-1},2]\times B(1)}v_k & \leq \int_{[T_{k-1}, 2]\times B(1)}v_{k-1} \chi_{\{v_{k-1} > \frac{1}{2^k} \}} \\
& \leq 2^{k( p-1 + \frac{p}{N} )} \int_{[T_{k-1}, 2]\times B(1)} v_{k-1}^{p(1+\frac{1}{N})} \\
& \leq C(N,p)   2^{k( p-1 + \frac{p}{N} )} U_{k-1}^{\frac{p}{N}} \big \{ U_{k-1} + U_{k-1}^{\frac{1}{p}}  \big \}^p
\end{split}
\end{equation}
\begin{equation}\label{raiseindextwo}
\begin{split}
\int_{[T_{k-1} , 2]\times B(1)} \chi_{\{v_k > 0 \}} & \leq \int_{[T_{k-1}, 2 ]\times B(1)} \chi_{\{ v_{k-1} > \frac{1}{2^k} \}} \\
& \leq 2^{k p (1+ \frac{1}{N})} \int_{[T_{k-1} , 2]\times B(1)} v_{k-1}^{p(1+ \frac{1}{N})} \\
& \leq 2^{kp (\frac{N+1}{N})} C(N,p) U_{k-1}^{\frac{p}{N}} \big \{ U_{k-1} + U_{k-1}^{\frac{1}{p}} \big \}^p
\end{split}
\end{equation}
So, it follows from \eqref{raiseprep}, \eqref{raiseindexONE}, and \eqref{raiseindextwo} that the following relation holds for each $k\geq 1$, where $C(N,p,\Lambda ) > 0$ is some absolute constant which depends only on $N$, $p$ and $\Lambda$.
\begin{equation}\label{nonlinearone}
U_k \leq C(N,p,\Lambda )^{k} U_{k-1}^{\frac{p}{N}} \big ( U_{k-1} + U_{k-1}^{\frac{1}{p}} \big )^p .
\end{equation}


Now, for some technical purpose, we now need to verify the following relation for all $k \geq 1$.
\begin{equation}\label{easyineq}
U_k \leq  2\Lambda \int_{[0,2]\times B(1)} v_k + \Lambda^2 \int_{[0,2]\times B(1)} \chi_{\{ v_k > 0 \}}
\end{equation}
In order to verify \eqref{easyineq} for each $k\geq 1$, we first recall that relation \eqref{good} holds for all variables $\sigma$, $t$ which satisfy the constraint $0 \leq \sigma \leq T_k \leq t \leq 2$. Thanks to the fact that $T_k \geq \frac{1}{2}$ holds for any $k\geq 1$, by taking the average over $\sigma \in [0, T_k]$ on each term of \eqref{good}, we easily yield the following estimate
\begin{equation*}
\int_{B(1)}v_k(t,\cdot ) + \frac{1}{\Lambda} \int_{T_k}^t \int_{B(1)} |\nabla v_k|^p \leq 2 \int_{[0,2]\times B(1)} v_k +
\Lambda \int_0^t \int_{B(1)} \chi_{\{ v_k > 0 \}} ,
\end{equation*}
which holds for $t \in [T_k , 2]$. So, by simply taking $\sup$ over $t \in [T_k , 2]$ of each term which appears in the above estimate, we immediately obtain \eqref{easyineq} for each $k \geq 1$, as desired.
Since it is obvious that the following estimate is valid for each $k \geq 1$
\begin{equation*}
\int_{[0,2]\times B(1)} \chi_{\{ v_k > 0\}} \leq 2 \int_{[0,2]\times B(1)} u_+ ,
\end{equation*}
it follows directly from \eqref{easyineq} that we have the following estimate for each $k \geq 1$
\begin{equation}\label{sortofimport}
U_k \leq 2 \big (  \Lambda + \Lambda^2   \big ) \int_{[0,2]\times B(1)} u_+ .
\end{equation}
\eqref{sortofimport} immediately leads to the following assertion
\begin{itemize}
\item If it happens that $\int_{[0,2]\times B(1)} u_+ < \frac{1}{2\Lambda (1+ \Lambda )}$, then it follows that $U_k < 1$ holds for all $k \geq 1$.
\end{itemize}
Due to the above assertion, we can now say that as long as $u$ satisfies $\int_{[0,2]\times B(1)} u_+ < \frac{1}{2\Lambda (1+ \Lambda )}$, it follows from \eqref{nonlinearone} that the following nonlinear recurrence relation holds for each $k \geq 1$.
\begin{equation}\label{nonlinearrecurrence}
U_k \leq D(N,p,\Lambda ) U_{k-1}^{1 + \frac{p}{N}} ,
\end{equation}
in which $D(N,p, \Lambda ) > 0$ again depends only on $N$, $p$ and $\Lambda$. In light of \eqref{nonlinearrecurrence}, it is time to recall the following well-known assertion of the De-Giorig's method.
\begin{itemize}
\item \textbf{Assertion I} Let $D(N,p, \Lambda ) > 0$ to be the absolute constant which appears in \eqref{nonlinearrecurrence}. Then, there exists some $\varepsilon_0 \in (0,1)$, which depends only on $D(N,p, \Lambda ) > 0$ and $1+ \frac{p}{N}$, such that for any sequence $\{a_k\}_{k=1}^{\infty}$ of nonnegative numbers for which $a_1 \leq \varepsilon_0$ holds and for which the relation $a_k \leq D(N,p,\lambda) a_{k-1}^{1+\frac{p}{N}}$ holds for all $k \geq 1$, it follows that $\lim_{k\rightarrow \infty} a_k = 0$.
\end{itemize}
In accordance with the above assertion, we now take
\begin{equation*}
\delta =  \frac{\varepsilon_0}{2\Lambda (1+ \Lambda )} .
\end{equation*}
Then, whenever $u: [0,2]\times B(1)\rightarrow \mathbb{R}$ is a solution to \eqref{basicineqONE} which satisfies $\int_{[0,2]\times B(1)}u_+ < \delta$, the associated sequence $U_k$ of truncated energies must satisfy both $U_1 < \varepsilon_0$ and \eqref{nonlinearrecurrence} for each $k\geq 1$, and hence it follows from \textbf{Assertion I} that $\lim_{k\rightarrow \infty}U_k = 0$. This immediately lead to the following conclusion:
\begin{itemize}
\item If $u :[0,2] \times B(1) \rightarrow \mathbb{R}$ is a solution to \eqref{basicineqONE} which satisfies $\int_{[0,2]\times B(1)} u_+ < \delta$, it follows that $u_+ \leq 1$ holds on $[1,2]\times B(1)$.
\end{itemize}
In other words, the proof of Lemma \ref{lemm:first} is now completed.
\end{proof}
\section{The second De-Giorgi's Lemma}\label{section:second}

We want to show now the following lemma.
\begin{lemm}\label{lemm:second}
Let $N\in \mathbb{Z}^+$, and $p \in (1,N)$ to be given. Then there exists some absolute constant $\alpha = \alpha (N,\Lambda , p )>0$, such that,
for any function $u: [-2,2]\times B(1)\rightarrow \mathbb{R}$ which is a weak solution (in the sense of distribution) to
$$
\dt u +\frac{1}{\Lambda}|\nabla u|^p\leq \Lambda, \qquad \mathrm{on} \ \  [-2,2]\times B(1),
$$
we have the following implication:\\
If it happens that $u\leq 2$ holds on $ [-2,2]\times B(1)$, and that $u$ satisfies the following two properties
\begin{eqnarray}
\label{eq:bas}
&& \big |\{ (t,x) \in [-2,2]\times B(1) : u(t,x)\leq 0\} \big |\geq \frac{|[-2,2]\times B(1)|}{2},\\
\label{eq:milieu}
&&\big |\{ (t,x) \in [-2,2]\times B(1) : 0 < u(t,x) < 1 \} \big |\leq \alpha,
\end{eqnarray}
then it follows that
$$
\int_{ [0,2]\times B(1)}[u-1]_+ < \frac{\delta}{2},
$$
where $\delta = \delta (N,\lambda , p) >0$ is the absolute constant whose existence is asserted in Lemma \ref{lemm:first}.
\end{lemm}
\begin{proof}
We divide the proof in several parts.
\vskip0.3cm
\noindent{\bf Step 1.} For any $u$ verifying (\ref{eq:bas}), we have the following relation for any $s,t\in [-2,2]$ with $s <t$ .
$$
\int_{B(1)}u_+(t,x)\,dx \leq \int_{B(1)}u_+(s,x)\,dx -\frac{1}{\La}\int_{s}^t\int_{B(1)} |\nabla u_+|^p\,dx\,ds+(t-s)\La |B(1)|.
$$
By taking $t=2, s=-2$ in the above estimate, we obtain the following estimate.
\begin{equation}\label{eq:gradient}
\int_{[-2,2]\times B(1)}|\nabla u_+|^p\,dx\,ds \leq C(\La).
\end{equation}
Note that the following estimate holds
$$
\dt u_+\leq \Lambda.
$$
So obviously, we have
$$
\|(\dt u_+)_+\|_{\mathcal{M}}\leq \Lambda |B(1)|,
$$
where $\|\cdot\|_\mathcal{M}$ stands for norm of measures in $[-2,2]\times B_1$.
But
\begin{equation}
\begin{split}
\int_{[-2,2]\times B(1)}|(\dt u_+)_-|\,dx\,dt & \leq -\int_{[-2,2]\times B(1)}\partial_t u_+ \,dx\,dt+\int_{[-2,2]\times B(1)}(\partial_t u_+)_+\,dx\,dt \\
& \leq 4|B(1)|+4\La |B(1)|.
\end{split}
\end{equation}
Hence
\begin{equation}\label{eq:dt}
\|\dt u\|_{\mathcal{M}}\leq C(\La).
\end{equation}

\vskip0.3cm
\noindent{\bf Step 2.}
Assume that Lemma \ref{lemm:second} is wrong. Then there exists a sequence of functions $\{u_k\}_{k=1}^{\infty}$ on $[-2,2]\times B(1)$, with each $u_k$ verifies (\ref{eq:gradient}), (\ref{eq:dt}) and the following three estimates.

\begin{equation}
\begin{split}
\int_{ [0,2]\times B(1)}[u_k-1]_+ & \geq \frac{\delta}{2}, \\
\big |\{ (t,x) \in [-2,2]\times B(1) : u_k(t,x)\leq 0\} \big | & \geq \frac{|[-2,2]\times B(1)|}{2}, \\
\big |\{ (t,x) \in [-2,2]\times B(1) : 0 < u_k(t,x) < 1\}  \big | & \leq \frac{1}{k}.
\end{split}
\end{equation}
From step 1, there exists a subsequence $\{u_{k_n}\}_{n=1}^{\infty}$ such that $u_{k_n}$ converges to $\bu$ in $L^1 ([-2,2]\times B(1))$, where the limiting function $\bu$ still verifies the following properties.
\begin{eqnarray*}
&&\bu\leq 2\qquad \mathrm{in}\ \ [-2,2]\times B(1),\\
&&\int_{B(1)}\bu_+(t,x)\,dx \leq \int_{B(1)}\bu_+(s,x)\,dx+(t-s)\La |B(1)|,\qquad -2\leq s\leq t\leq 2,\\
&&\int_{[-2,2]\times B(1)}|\nabla \bu_+|^p\,dx\,ds \leq C(\La),\\
&& \int_{[0,2]\times B(1)}[\bu-1]_+\geq \frac{\delta}{2}.
\end{eqnarray*}
Observe that, the following estimate holds for any $\eps>0$,
$$
\big |\{ |u_{k_n}-\bu|\geq\eps \}  \big |\leq \frac{1}{\eps}\int_{[-2,2]\times B(1)}|u_{k_n}-\bu|,
$$
So, it follows that for each fixed $\eps > 0$, the term which appears in the left hand side of the above estimate converges to 0 when $k_n$ goes to infinity.
Now we have the following obvious estimate.
\begin{equation}
\begin{split}
\big |\{\bu\leq\eps\} \big | & \geq \big |\{u_{k_n}\leq 0\}  \big |-
\big |\{|u_{k_n}-\bu|\geq\eps\}  \big |\\
& \geq \frac{\big |[-2,2]\times B(1)\big | }{2}- \big |\{|u_{k_n}-\bu|\geq\eps\} \big |.
\end{split}
\end{equation}
Through passing into the limit on the above estimate as $k_n$ goes to infinity, we get
$$
\big |\{\bu\leq\eps\} \big |\geq \frac{ \big | [-2,2]\times B(1) \big |}{2} ,
$$
which is true for any $\eps>0$. So, by passing to the limit on the above estimate by letting $\eps \rightarrow 0^+$, it follows that we have
$$
\big | \{\bu\leq0\} \big | \geq \frac{\big | [-2,2]\times B(1) \big |}{2}.
$$
In the same way, we have,
\begin{equation*}
\begin{split}
\big | \{\eps\leq \bu\leq1-\eps\} \big | & \leq \big | \{0<u_{k_n}<1\} \big |+ \big | \{|u_{k_n}-\bu|\geq\eps\}\big | \\
& \leq \frac{1}{k_n} + \frac{1}{\eps}\int_{[-2,2]\times B(1)}|u_{k_n}-\bu| .
\end{split}
\end{equation*}
By passing to the limit on the above estimate as $k_n\rightarrow \infty$, we can deduce that the following relation holds for every $\eps > 0$
\begin{equation*}
\big | \{\eps\leq \bu\leq1-\eps\} \big | = 0 ,
\end{equation*}
from which it follows, through taking $\eps \rightarrow 0^+$, that the following relation holds.

\begin{equation}\label{interesting}
 \big |\{ (t,x) \in [-2,2]\times B(1) : 0 < \bu (t,x) <1 \} \big | =0.
\end{equation}

\vskip0.3cm
\noindent{\bf Step 3.}
Now, we observe that, for almost every $t\in [-2,2]$ ,
$$
\int_{B(1)}|\nabla \bu|^p(t,x)\,dx
$$
is finite. Also, \eqref{interesting} tells us that the following relation holds for almost every $t\in [-2,2]$
$$
\big |\{0< \bu(t,\cdot)<1\}\cap B(1) \big | = 0,
$$
So, by an application of the isoperimetric lemma of De Girogi (with fixed time $t$), it follows that, for almost every $t\in [-2,2]$, we have either
\begin{eqnarray*}
&&\bu(t,\cdot)\leq 0\qquad \mathrm{in} \ \ B(1),\\
&& \mathrm{or}\qquad \qquad \bu(t,\cdot)\geq1\qquad \mathrm{in} \ \ B(1).
\end{eqnarray*}
Especially, for almost every $t \in [-2,2]$, we have either
$$
\int_{B(1)}\bu_+(t,x)\,dx=0
$$
or
$$
\int_{B(1)}\bu_+(t,x)\,dx\geq |B(1)|.
$$
\vskip0.3cm
\noindent{\bf Step 4.}
Since
$$
 |\{\bu\leq0\}|\geq \frac{\big |[-2,2]\times B(1) \big| }{2},
$$
there exists $s \in [-2,0]$ for which the following relation holds
$$
\int_{B(1)}\bu_+(s,x)\,dx=0.
$$
Consider $s_0\in [s,2]$ to be the supremum of all such time $\bar{s} \in [s,2]$ which satisfies the following property
$$
\int_{B(1)}\bu_+ (\tau ,\cdot )\,dx=0 \qquad \mathrm{for} \ \ \tau \in [s,\bar{s}).
$$
If it happens that $s_0<2$, then for $t\in (s_0,s_0+ \frac{1}{2\Lambda}] \cap [-2,2]$, we have
$$
\int_{B(1)}\bu_+(t,x)\,dx\leq \Lambda |B(1)| (t-s_0)\leq \frac{|B(1)|}{2}.
$$
But then the above estimate will lead to the following relation
$$
\int_{B(1)}\bu_+(t,x)\,dx=0 ,
$$
which holds for all $t\in (s_0,s_0+ \frac{1}{2\Lambda}] \cap [-2,2]$. This directly contradicts the definition of $s_0$. This means that we have no choice but to admit that $s_0$ must be $2$. However, $s_0 = 2$ would mean that $\bu \leq 0$ holds on $[s,2]\times B(1)$, which however directly contradicts the fact that we should have

$$
 \int_{[0,2]\times B(1)}[\bu-1]_+\geq \frac{\delta}{2}.
$$
This ends the proof.
\end{proof}
\section{Improved oscillations from above and below}\label{section:oscillation}

Let $\delta = \delta (N, \Lambda , p) > 0$, and $\alpha = \alpha (N, \Lambda , p) >0 $ to be the two absolute constants which are specified in Lemma \ref{lemm:first} and Lemma \ref{lemm:second} respectively. Now, we consider the integer $K_0 \in \mathbb{Z}^+$ which is defined as follows.

\begin{equation}\label{specialinteger}
K_0 = \bigg [ \frac{\big |[-2,2]\times B(1) \big |}{\alpha } \bigg ] + 1 ,
\end{equation}
where the symbol $[x]$ means the largest integer which is less than or equal to $x$. It is obvious that $K_0$ is an absolute constant which depends only on $N$, $\Lambda$, and $p$, since both $\delta$ and $\alpha$ do. Now, we consider the following two differential inequalities.
\begin{equation}\label{rescaledONE}
\dt u + \frac{2^{(K_0 + 1)(p-1)}}{\Lambda} \big |\nabla u \big |^p \leq \frac{\Lambda}{2^{K_0 +1}} .
\end{equation}
\begin{equation}\label{rescaledTWO}
\dt u + 2^{(K_0+1) (p-1)} \Lambda  \big |\nabla u \big |^p \geq 0 .
\end{equation}

Now, by applying Lemma \ref{lemm:second}, and then Lemma \ref{lemm:first} successively, we can now obtain the following proposition.

\begin{prop}\label{prop:above}
\textbf{(Improved oscillation from above)} There exists some absolute constant $\lambda = \lambda (N,\Lambda , p ) \in (0,1)$, such that for any weak solution $u :[-2, 2]\times B(1)\rightarrow \mathbb{R}$ to \eqref{rescaledONE} on $[-2,2]\times B(1)$, we have the following implication.\\

If it happens that $u \leq 2$ holds on $ [-2,2]\times B(1)$ and that $u$ satisfies the following property
$$
\big |\{u(t,x)\leq 0\} \cap ([-2,2]\times B(1)) \big |\geq \frac{\big |[-2,2]\times B(1) \big |}{2},
$$
then it follows that
$$
u \leq 2-\lambda \qquad \mathrm{on}\ \ \ [1,2]\times B(1).
$$
\end{prop}
\begin{proof}
Let $u :[-2,2 ]\times B(1)\rightarrow \mathbb{R}$ to be a weak solution to \eqref{rescaledONE} which satisfies all the hypothesis of Proposition \ref{prop:above}. For each integer $k\in \mathbb{Z}^+$ which satisfies $1 \leq k \leq K_0 + 1$, we consider the function $u_k : [-2,2]\times B(1) \rightarrow \mathbb{R}$ which is defined through the following relation in an inductive manner.
\begin{equation*}
u_k = 2 (u_{k-1} - 1 ) ,
\end{equation*}
where the function $u_0$ is just defined to be $u$ itself (i.e. $u_0 = u$). Then, inductively, it is easy to see that the following identity holds for each $k \in \{ 1,2,...K_0+1\}$.
\begin{equation*}
u_k = 2^k \big \{  u - 2 (1-\frac{1}{2^k})  \big \} .
\end{equation*}
By construction, it is obvious that, for each $1 \leq k \leq K_0 + 1 $ the relation $u_k \leq 2$ holds on $[-2,2]\times B(1)$.
Observe that the following relation holds for each $k \in \{1,2,...,K_0\}$.
\begin{equation}\label{usefulone}
\begin{split}
\big | \{(t,x) \in [-2,2]\times B(1) : u_k(t,x) \leq 0 \} \big | & \geq \big | \{(t,x) \in [-2,2]\times B(1) : u(t,x) \leq 0 \} \big | \\
& \geq \frac{\big | [-2,2]\times B(1) \big |}{2} .
\end{split}
\end{equation}
Inductively, it is apparent that each $u_k$ is a weak solution to the following differential inequality on $[-2,2]\times B(1)$
\begin{equation}\label{rescaledweaken1}
\dt u_k + \frac{2^{(p-1)(K_0 +1 -k)}}{\Lambda} \big | \nabla u_k\big |^p \leq \frac{\Lambda}{2^{K_0 + 1 - k}} .
\end{equation}
Next, we note that it is \emph{impossible} to have the following relation to be valid for all $k \in \{1,2,...,K_0\}$
\begin{equation*}
\big | \{ (t,x) \in [-2,2]\times B(1) : 0 < u_k(t,x) <1   \} \big | > \alpha ,
\end{equation*}
since if otherwise, the validity of the above relation for all $1 \leq k \leq K_0$ would lead to
\begin{equation*}
\begin{split}
\big |\{(t,x) \in [-2,2]\times B(1) : u_{K_0}(t,x) \leq 0 \} \big | & > \frac{\big | [-2,2]\times B(1)\big |}{2} + K_0 \alpha \\
& > \frac{3}{2} \big | [-2,2]\times B(1)\big | ,
\end{split}
\end{equation*}
which is absurd. This indicates that there must be some positive integer $j_0$ which satisfies $1 \leq j_0 \leq K_0$ for which the following relation holds.
\begin{equation}\label{usefultwo}
\big |\{(t,x) \in [-2,2]\times B(1) : 0 < u_{j_0}(t,x) <1 \} \big | \leq \alpha .
\end{equation}
Since $K_0 - j_0 \geq 0$, it follows that $u_{j_0}$ is also a weak solution to \eqref{basicineqONE} on $[-2,2]\times B(1)$. As a result, \eqref{usefulone} and \eqref{usefultwo} together enable us to apply Lemma \ref{lemm:second} directly to $u_{j_0}$ in order to deduce that the following property holds.
\begin{equation}\label{final}
\int_{[0,2]\times B(1)} (u_{j_0 +1})_{+} = 2 \int_{[0,2]\times B(1)} \big ( u_{j_0} - 1 \big )_{+} \leq \delta .
\end{equation}
Since $u_{j_0 +1}$ satisfies \eqref{basicineqONE}, \eqref{final} enables us to apply Lemma \ref{lemm:first} directly to $u_{j_0 +1}$ to deduce that
$$
u_{j_0 + 1} \leq 1 \qquad \mathrm{on}\ \ \ [1,2]\times B(1),
$$
which gives
$$
u \leq 2 - \frac{1}{2^{j_0 +1}} \leq 2 - \frac{1}{2^{K_0 +1}} \qquad \mathrm{on}\ \ \ [1,2]\times B(1).
$$
So, by taking $\lambda = \frac{1}{2^{K_0 +1}}$, the proof of Proposition \ref{prop:above} is complete.
\end{proof}

Now we want to show the following proposition.
\begin{prop}\label{prop:below}\textbf{(Improved oscillation from below)} There exists absolute constants $\tilde{\lambda} = \tilde{\lambda}(N,\Lambda , p) \in (0,1)$, and $q = q (N, \Lambda , p)>0$ such that for any function $u : [-2, 2]\times \mathbb{R}^N \rightarrow \mathbb{R}$ which is a weak solution to \eqref{rescaledONE} on $[-2,2]\times B(1)$, and which simultaneously is a viscosity solution to \eqref{rescaledTWO} on $[-2,2]\times \mathbb{R}^N$, we have the following implication:\\
If $u$ satisfies the following properties
\begin{equation}\label{usuallower}
u  \geq  -2  \qquad \mathrm{on}\ \ \ [-2,2]\times B(1) .
\end{equation}
\begin{equation}\label{GOODARRAY}
\big |\{u(t,x)\geq 0\} \cap ([-2,2]\times B(1)) \big|\geq \frac{\big |[-2,2]\times B(1) \big |}{2},
\end{equation}
\begin{equation}\label{globallowerbarrier}
u(t,x)\geq -2-q(|x|-1)_+, \qquad {for} \ \ (t,x) \in [-2,2]\times \R^N,
\end{equation}
then it follows that
$$
u  \geq  -2+\tilde{\lambda} \qquad \mathrm{on}\ \ \ [1,2]\times B(\frac{1}{2}).
$$
\end{prop}
\begin{proof}
Let $u : [-2,2]\times \mathbb{R}\rightarrow \mathbb{R}$ to be a viscosity solution to \eqref{rescaledTWO} on $[-2,2]\times \mathbb{R}^N$ which is also a weak solution to \eqref{rescaledONE} on $[-2,2]\times B(1)$ and which satisfies conditions \eqref{usuallower} and \eqref{GOODARRAY}. Condition \eqref{globallowerbarrier} will eventually be imposed on our solution $u$. But we do not do so at this moment, simply due to the fact that the absolute constant $q= q (N,\Lambda , p) > 0$ is not specified yet.

The function $v: [-2,2]\times \mathbb{R}^N\rightarrow \mathbb{R}$ defined by
$$
v(t,x)=-u(-t,x)
$$
is not a viscosity solution to \eqref{rescaledTWO} anymore, but it still verifies \eqref{rescaledONE} in the sense of distribution.
Condition \eqref{GOODARRAY} is equivalent to
\begin{equation}\label{aboutvone}
\big |\{v(t,x)\leq 0\} \cap ([-2,2]\times B(1)) \big|\geq \frac{\big |[-2,2]\times B(1) \big |}{2}.
\end{equation}
Also, condition \eqref{usuallower} gives that $v \leq 2$ holds on $[-2,2]\times B(1)$.
Hence, we can directly apply Proposition \ref{prop:above} to deduce that,
$$
v\leq 2-\lambda \qquad \mathrm{on} \ \ [1,2]\times B(1),
$$
which is equivalent to,
$$
u\geq-2+\lambda\qquad  \mathrm{on} \ \  [-2,-1]\times B(1).
$$
Here, $\lambda \in (0,1)$ is the absolute constant which is specified in proposition \ref{prop:above}.
With respect to some positive number $\lambda_1 \in (0, \lambda )$ which will be determined later, consider the function $\psi : [-2,2]\times \mathbb{R}^N \rightarrow \mathbb{R}$ which is defined as follows.
$$
\psi(t,x)=\inf\left(-2+\lambda_1; -2-\frac{\lambda_1}{8}(t+2) + \left(\frac{\lambda_1}{8\La \cdot 2^{(p-1)(K_0 + 1)}}\right)^{1/p}\big (1-|x| \big )\right).
$$
The function $\psi$ is a viscosity solution to
$$
\dt \psi +2^{(K_0 +1)(p-1)}  \Lambda|\nabla \psi |^p=0, \qquad \mathrm{on} \ \  [-2,2]\times \mathbb{R}^N.
$$
Now, suppose further that $u$ satisfies the additional condition \eqref{globallowerbarrier} with $q$ to be defined as follows.
\begin{equation}\label{q}
q = \left(\frac{\lambda_1}{8\La \cdot 2^{(p-1)(K_0 + 1)}}\right)^{1/p} .
\end{equation}
Then, it follows that the relation $\psi(-2,\cdot )\leq u(-2,\cdot )$ holds on $\R^N$. Hence, in accordance with basic comparison principle in the theory of viscosity solutions, it follows that the following relation holds for all $(t,x)\in [-2,2]\times \mathbb{R}^N$
$$
u(t,x)\geq \psi(t,x),
$$
from which it follows that the following relation holds
\begin{equation}
u \geq \inf \bigg ( -2 + \lambda_1 ;  -2 - \frac{\lambda_1}{2} + \frac{1}{2}\left(\frac{\lambda_1}{8\La \cdot 2^{(p-1)(K_0 + 1)}}\right)^{1/p} \bigg ) \qquad \mathrm{on} \ \  [1,2]\times B(\frac{1}{2}) .
\end{equation}
In the case of $p \in (1, N)$, since
\begin{equation*}
\lim_{\lambda_1 \rightarrow 0^+} \left \{ \left(\frac{\lambda_1}{8\La \cdot 2^{(p-1)(K_0 + 1)}}\right)^{1/p}\cdot \frac{1}{\lambda_1} \right \} = + \infty ,
\end{equation*}
it follows that we may choose $\lambda_1 \in (0, \lambda )$ to be sufficiently small so that the following relation holds
\begin{equation}\label{constraintonlambda}
2\lambda_1 <   \left(\frac{\lambda_1}{8\La \cdot 2^{(p-1)(K_0 + 1)}}\right)^{1/p} .
\end{equation}
With respect to such a $\lambda_1 \in (0,\lambda )$ satisfying \eqref{constraintonlambda}, we deduce that the following improved oscillation holds, provided $u$ satisfies \eqref{GOODARRAY}, and \eqref{globallowerbarrier} with $q > 0$ to be the constant specified in \eqref{q}.
\begin{equation*}
u \geq \inf \big ( -2 + \lambda_1 ; -2 + \frac{\lambda_1}{2} \big ) = -2 + \frac{\lambda_1}{2} \qquad \mathrm{on} \ \  [1,2]\times B(\frac{1}{2}) .
\end{equation*}
So, by taking $\tilde{\lambda} = \frac{\lambda_1}{2}$,
the proof of proposition \ref{prop:below} is now completed.

\end{proof}
\section{Final rescaling}\label{section:final}

Indeed, by means of a simple re-scaling argument, it is easy to see that Theorem \ref{theo:main} will follow from the following proposition, which we are going to prove in this final section.

\begin{prop}\label{last}
Consider $u :[-4,0]\times \mathbb{R}^N \rightarrow \mathbb{R}$ to be a weak solution to \eqref{basicineqONE} on $[-4,0]\times B(1)$, which simultaneously is a viscosity solution to \eqref{superviscosity} on $[-4,0]\times \mathbb{R}^N$. suppose that $|u| \leq 2$ holds on $[-4,0]\times \mathbb{R}^N$. Then, it follows that $u(0, \cdot ) \in C^{\alpha}(\mathbb{R}^N)$ holds, with some $\alpha \in (0,1)$ which depends only on $N$, $\Lambda$, and $p$.
\end{prop}

\begin{proof}
The proof of proposition \ref{last} will be carried out through several steps as follows.\\

\noindent{\bf Step 1: Initial re-scaling}\\

Let $u :[-4,0]\times \mathbb{R}^N \rightarrow \mathbb{R}$ to be a function which is a weak solution to \eqref{basicineqONE} on $[-4,0]\times B(1)$, and which simultaneously is a viscosity solution to \eqref{superviscosity} on $[-4,0]\times \mathbb{R}^N$. We will assume, without the loss of generality, that $|u| \leq 2$ holds on $[-4,0]\times \mathbb{R}^N$. In order to use either improved oscillation from above or below by means of Propositions \ref{prop:above} or \ref{prop:below}, we need to re-scale our function $u$ so that the re-scaled function will satisfies \eqref{rescaledONE} in the weak sense and \eqref{rescaledTWO} in the viscosity sense. Moreover, with respect to some parameters $\varepsilon \in (0,1)$, and $\alpha \in [1, p )$ which will be determined later, we consider the function $u_1 : [\frac{-4}{\varepsilon^{\alpha} } , 0 ]\times \mathbb{R}^N \rightarrow \mathbb{R}$ defined as follows (This index $\alpha \in [1,p)$ \emph{absolutely has nothing to do} with the previous absolute constant $\alpha(N, \Lambda, p) > 0$ which appears in Lemma \ref{lemm:second}).
\begin{equation*}
u_1(t,x) = u(\varepsilon^{\alpha} t , \varepsilon x ) .
\end{equation*}
Notice that $u_1$ is then a weak solution to the following inequality on $ [\frac{-4}{\varepsilon^{\alpha} } , 0 ]\times B(\frac{1}{\varepsilon}) $
\begin{equation}\label{U1weak}
\dt u_1 + \frac{1}{\Lambda \varepsilon^{p-\alpha } } \big | \nabla u  \big |^p \leq \Lambda \varepsilon^{\alpha}.
\end{equation}
At the same time $u_1$ is also a viscosity solution to the following inequality on $ [\frac{-4}{\varepsilon^{\alpha} } , 0 ]\times \mathbb{R}^N$.
\begin{equation}\label{U1viscosity}
\dt u_1 + \frac{\Lambda}{\varepsilon^{p-\alpha }}  \big | \nabla u  \big |^p \geq 0 .
\end{equation}
Now, we simply notice that
\begin{equation*}
\big \{\frac{1}{\varepsilon^{p-\alpha}} : \varepsilon \in (0,1) , \alpha \in [1, p ) \big \} = (1, + \infty ) ,
\end{equation*}
which immediately ensures that we can find a suitable pair $(\varepsilon , \alpha )$ with $\varepsilon \in (0,1)$ and $\alpha \in [1, p )$ for which the following required property holds
\begin{equation}\label{required}
\frac{1}{\varepsilon^{p-\alpha }} = 2^{(K_0+1)(p-1)} .
\end{equation}
So, with respect to such a pair of $(\varepsilon , p)$ satisfying \eqref{required}, since $\frac{p-1}{p-\alpha} \geq 1$, we clearly have
\begin{equation*}
\varepsilon^{\alpha} = \frac{1}{2^{(K_0 +1)(\frac{p-1}{p-\alpha}) \alpha }} \leq \frac{1}{2^{K_0 +1 }} .
\end{equation*}
Hence, it follows from \eqref{U1weak} and \eqref{U1viscosity} that $u_1$ is a weak solution to \eqref{rescaledONE} on $ [\frac{-4}{\varepsilon^{\alpha} } , 0 ]\times B(\frac{1}{\varepsilon}) $, which simultaneously is also a viscosity solution to \eqref{rescaledTWO} on $ [\frac{-4}{\varepsilon^{\alpha} } , 0 ]\times \mathbb{R}^N$. Since we still have $|u_1| \leq 2$ on $ [\frac{-4}{\varepsilon^{\alpha} } , 0 ]\times \mathbb{R}^N$, we can directly apply either Propositions \ref{prop:above} or \ref{prop:below} to yield either $u_1 \leq 2 - \tilde{\lambda}$ on
$[-1,0]\times B(\frac{1}{2})$ or $u_1 \geq -2 + \tilde{\lambda} $ on $[-1 ,0]\times B(\frac{1}{2})$.\\
In either case, we obtain the following improved oscillation
\begin{equation}
\osc_{[-1,0]\times B(\frac{1}{2})} u_1 \leq 4 - \tilde{\lambda}.
\end{equation}
\noindent{\bf Step 2: second re-scaling and improved oscillation at the scale of $\varepsilon_1$ }\\

The situation looks so far so good. However, we have to be more careful in carrying out the second re-scaling. First, due to the above improved oscillation of $u_1$ on $[-1, 0]\times B(\frac{1}{2})$, we can find some $d_1 \in \mathbb{R}$ with $|d_1| \leq \frac{\tilde{\lambda}}{2}$ such that the following relation $|u_1 - d_1| \leq 2 - \frac{\tilde{\lambda}}{2}$ holds on $[-1,0]\times B(\frac{1}{2})$. Just as before, with respect to some
$\varepsilon_1 \in (0,1)$ and $\alpha_1 \in [1 ,p )$ which have to be determined later, we need to consider the re-scaled function
$u_2$ defined as follows

\begin{equation}\label{utwo}
u_2 (t,x) = \frac{4}{4-\tilde{\lambda}}\bigg \{ u_1(\varepsilon_1^{\alpha_1} t , \varepsilon_1 x ) - d_1 \bigg \} .
\end{equation}

Observe that $u_2$ is then a weak solution to the following inequality on $ [\frac{-4}{\varepsilon^{\alpha} \varepsilon_1^{\alpha_1} } , 0 ] \times    B(\frac{1}{\varepsilon \varepsilon_1})$
\begin{equation}\label{reallyGOOD}
\dt u_2 + \frac{2^{(K_0 + 1) (p-1)} }{\Lambda \varepsilon_1^{p-\alpha_1}}\big ( \frac{4-\tilde{\lambda}}{4} \big )^{p-1} \big | \nabla u_2\big |^p \leq \frac{\Lambda}{2^{K_0 +1}} \epsilon_1^{\alpha_1} \big ( \frac{4}{4-\tilde{\lambda}} \big ).
\end{equation}


Also, $u_2$ is a viscosity solution to the following inequality on $ [\frac{-4}{\varepsilon^{\alpha} \varepsilon_1^{\alpha_1} } , 0 ] \times \mathbb{R}^N$
\begin{equation}\label{reallyBAD}
\dt u_2 + \Lambda 2^{(K_0 +1) (p-1)} \frac{1}{\varepsilon_1^{p-\alpha_1 }}\big ( \frac{4-\tilde{\lambda}}{4} \big )^{p-1} \big | \nabla u_2\big |^p \geq 0 ,
\end{equation}


In light of the structure of \eqref{reallyBAD}, it is natural to take $\varepsilon_1$ as follows, with some suitable $\alpha_1 \in (1,p)$ which will be determined later.
\begin{equation}\label{goodepsilonONE}
\varepsilon_1 = \bigg ( \frac{4-\tilde{\lambda}}{4}\bigg )^{\frac{p-1}{p-\alpha_1}}
\end{equation}
With this choice of $\epsilon_1$ as specified in \eqref{goodepsilonONE}, we surely have
\begin{equation*}
\frac{1}{\varepsilon_1^{p-\alpha }}\big ( \frac{4-\tilde{\lambda}}{4} \big )^{p-1} = 1 .
\end{equation*}
and that
\begin{equation*}
\epsilon_1^{\alpha_1} \big ( \frac{4}{4-\tilde{\lambda}} \big ) = \bigg ( \frac{4-\tilde{\lambda}}{4} \bigg )^{\frac{p(\alpha_1 -1)}{p-\alpha_1}} < 1 ,
\end{equation*}
since $\alpha_1 \in (1, p)$ ensures that $\frac{p(\alpha_1 -1)}{p-\alpha_1} > 0$.
These observations tell us that, as long as $\epsilon_1$ is given by \eqref{goodepsilonONE}, it follows from \eqref{reallyGOOD} and \eqref{reallyBAD} that $u_2$ is still a weak solution to \eqref{rescaledONE} on $ [\frac{-4}{\varepsilon^{\alpha} \varepsilon_1^{\alpha_1} } , 0 ] \times    B(\frac{1}{\varepsilon \varepsilon_1})$, which simultaneously is also a viscosity solution to \eqref{rescaledTWO} on $ [\frac{-4}{\varepsilon^{\alpha} \varepsilon_1^{\alpha_1} } , 0 ] \times \mathbb{R}^N$. Next, we have to choose a suitable $\alpha_1 \in (1,p)$, in a way which depends only on $N, \Lambda$ and $p$, to ensure that $|u_2|$ is entirely within the barrier function $2+ q(|x| - 1)_+  $ (Recall that $q = q(N, \Lambda , p ) > 0$ is the absolute constant which appears in proposition \ref{prop:below} ). \\

Since, due to construction, $u_2$ must satisfies $|u_2| \leq 2$ on $[-\frac{1}{\varepsilon_1^{\alpha_1}} , 0] \times B(\frac{1}{2\varepsilon_1})$, we only need to check that the following property really holds on $[-\frac{1}{\varepsilon_1^{\alpha_1}} , 0] \times \mathbb{R}^N$
\begin{equation}\label{check1}
|u_2(t,x)| \chi_{\{ |x| \geq \frac{1}{2 \varepsilon_1} \}} \leq  \big ( 2+ q(|x| - 1)_+ \big ) \chi_{\{ |x| \geq \frac{1}{2 \varepsilon_1}\} }.
\end{equation}
However, from \eqref{utwo}, it is clear that the following relation holds on $[-\frac{1}{\varepsilon_1^{\alpha_1}} , 0] \times \mathbb{R}^N$
\begin{equation*}
|u(t,x)| \chi_{\{|x| \geq \frac{1}{2 \varepsilon_1} \}} \leq \bigg ( \frac{4}{4-\tilde{\lambda} }\bigg )\bigg \{ 2 + \frac{\tilde{\lambda}}{2} \bigg \} \chi_{\{|x| \geq \frac{1}{2 \varepsilon_1} \}}
\end{equation*}
So, in order to ensure the survival of \eqref{check1}, it is enough to check that the following estimate holds
\begin{equation}\label{check2}
\bigg ( \frac{4}{4-\tilde{\lambda} }\bigg )\bigg \{ 2 + \frac{\tilde{\lambda}}{2} \bigg \} \chi_{\{|x| \geq \frac{1}{2 \varepsilon_1} \}}
\leq \big ( 2+ q(|x| - 1)_+ \big ) \chi_{\{ |x| \geq \frac{1}{2 \varepsilon_1}\} }.
\end{equation}
Since $2+ q(|x| -1)_+$ is monotone increasing in $|x|$, \eqref{check2} is valid if and only if the following relation holds
\begin{equation}\label{check3}
\begin{split}
\bigg ( \frac{4}{4-\tilde{\lambda} }\bigg )\bigg \{ 2 + \frac{\tilde{\lambda}}{2} \bigg \} & \leq 2 + q \bigg ( \frac{1}{2 \varepsilon_1} -1 \bigg ) \\
& = 2 + q \bigg ( \frac{1}{2} \big (\frac{4}{4-\tilde{\lambda}} \big )^{\frac{p-1}{p-\alpha_1}} - 1 \bigg ) .
\end{split}
\end{equation}
Notice that the last equal sign is due to \eqref{goodepsilonONE} which relates $\varepsilon_1 \in (0,1)$ to $\alpha_1 \in [1,p)$. The key point is that we have the freedom to choose $\alpha_1 \in (1,p)$ to be as close to $p > 1$ as possible. Indeed, we observe that
\begin{equation*}
\lim_{\alpha_1 \rightarrow p^-} \big (\frac{4}{4-\tilde{\lambda}} \big )^{\frac{p-1}{p-\alpha_1}} = + \infty ,
\end{equation*}
which allows us to choose some suitable $\alpha_1 \in (1,p)$ to be sufficiently close to $p$, in a manner which depends only on $p$, $q$ and $\tilde{\lambda}$ (and hence only on $N$, $p$, and $\lambda$) such that relation \eqref{check3} must be valid and that the following relation must hold simultaneously
\begin{equation}\label{Check4}
\frac{1}{\varepsilon_1^{\alpha_1}} = \big (\frac{4}{4-\tilde{\lambda}} \big )^{\frac{p-1}{p-\alpha_1}\alpha_1} > 4 .
\end{equation}
Notice that \eqref{Check4} is here to ensure that $[-4,0] \subset [- \varepsilon_1^{-\alpha_1} , 0]$ holds.
Hence with such a $\alpha_1 = \alpha_1 (N,p, \Lambda )$ sufficiently close to $p$, relation \eqref{check1} must be valid, and hence the desired relation $|u_2| \leq 2 + q(|x| - 1 )_+ $ holds on $[-4 , 0]\times \mathbb{R}^N$. As a result, with respect to such a $\alpha_1 = \alpha_1 (N,p,\Lambda ) \in (0,p)$, and $\varepsilon_1$ to be given in \eqref{goodepsilonONE}, we may apply either Propositions \ref{prop:above} or \ref{prop:below} to deduce that we can find some suitable $d_2 \in \mathbb{R}$ with $|d_2| \leq \frac{\tilde{\lambda}}{2}$ for which we have
$$
\big | u_2 - d_2 \big | \leq 2 - \frac{\tilde{\lambda}}{2} \qquad \mathrm{on}\ \ \ [-1,0]\times B(\frac{1}{2}).
$$

\noindent{\bf Step 3: Successive re-scalings and improved oscillations at finer and finer scales}\\
Let $\alpha_1 \in (1,p)$ to be the same absolute constant in \textbf{Step 2} which satisfies both \eqref{check3} and \eqref{Check4}, and let $\varepsilon_1 \in (0,1)$ to be the one specified in \eqref{goodepsilonONE}. Inductively, suppose that we have done re-scalings on the original solution $u : [-4,0]\times \mathbb{R}^N \rightarrow \mathbb{R}$ for $m-1$-times; so that we have found a list of numbers $d_1, d_2 , d_3,...d_{m-1} \in \mathbb{R}$, each of them satisfies $|d_j| \leq \frac{\tilde{\lambda}}{2}$, in such a way that the associated list of re-scaled functions $u_j$ as determined by the following recurrence relation (for $2 \leq k \leq m$)
\begin{equation}\label{recurrencerelation}
u_k(t,x) = \frac{4}{4-\tilde{\lambda}} \bigg \{ u_{k-1} (\varepsilon_1^{\alpha_1}t , \epsilon_1 x ) - d_{k-1}   \bigg \}
\end{equation}
satisfy the following properties.
\begin{itemize}
\item For each $1 \leq j \leq m$, $u_j$ is a weak solution to \eqref{rescaledONE} on $[-4,0]\times B(1)$, and simultaneously a viscosity solution to \eqref{rescaledTWO} on $[-4 , 0]\times \mathbb{R}^N$.
\item For each $1 \leq j \leq m$, the relation $|u_j(t,x)| \leq 2 + q(|x| - 1 )_+$ holds for all $(t,x) \in [-4,0]\times \mathbb{R}^N$.
\end{itemize}
Since $u_m$ satisfies the above two properties, we apply either propositions \ref{prop:above} or \ref{prop:below} to deduce that either $u_m \leq 2-\tilde{\lambda}$ or else $u_m \geq -2 + \tilde{\lambda}$ must hold on $[-1, 0 ]\times B(\frac{1}{2})$. In other words, we can find some suitable $d_m \in \mathbb{R}$ with $|d_m| \leq \frac{\tilde{\lambda}}{2}$ for which we have
$$
\big | u_m - d_m \big | \leq 2 - \frac{\tilde{\lambda}}{2} \qquad \mathrm{on}\ \ \ [-1,0]\times B(\frac{1}{2}).
$$
Now, consider the re-scaled function $u_{m+1} : [-\frac{4}{\varepsilon_1^{\alpha_1}}, 0]\times \mathbb{R}^N \rightarrow \mathbb{R}$ defined as follows
\begin{equation*}
u_{m+1}(t,x) = \frac{4}{4-\tilde{\lambda}} \bigg \{ u_{m} (\varepsilon_1^{\alpha_1}t , \epsilon_1 x ) - d_{m}   \bigg \}.
\end{equation*}
$u_{m+1}$ is then a weak solution to \eqref{rescaledONE} on $ [-\frac{4}{\varepsilon_1^{\alpha_1}}, 0]\times B(\frac{1}{\varepsilon_1})$, and simultaneously a viscosity solution to \eqref{rescaledTWO} on $[-\frac{4}{\varepsilon_1^{\alpha_1}}, 0]\times \mathbb{R}^N$.
By construction, we also have
$$
\big | u_{m+1} \big | \leq 2 \qquad \mathrm{on}\ \ \ [-\frac{1}{\varepsilon_1^{\alpha_1}},0]\times B(\frac{1}{2\varepsilon_1}),
$$
and that
$$
\big |u_{m+1}(t,x) \big | \leq  \frac{4}{4-\tilde{\lambda}}\bigg \{ 2+ q(\varepsilon_1 |x| -1)_+ + \frac{\tilde{\lambda}}{2} \bigg \}
\qquad \mathrm{on}\ \ \ [-\frac{4}{\varepsilon_1^{\alpha_1}}, 0]\times \mathbb{R}^N .
$$
So, to show that $|u_{m+1}(t,x)| \leq 2 + q(|x|-1)_+$ holds on $[-4,0]\times \mathbb{R}^N$, it suffices to show that the following relation holds on
$\mathbb{R}^N$
\begin{equation}\label{keyestimatefinal}
\frac{4}{4-\tilde{\lambda}}\bigg \{ 2+ q(\varepsilon_1 |x| -1)_+ + \frac{\tilde{\lambda}}{2} \bigg \} \chi_{\{|x| \geq \frac{1}{2 \varepsilon_1}\}}
\leq \big \{ 2 + q (|x| - 1)_+ \big \} \chi_{\{|x| \geq \frac{1}{2 \varepsilon_1}\}} .
\end{equation}
Thanks to the fact that

\begin{equation*}
\bigg ( \frac{4}{4-\tilde{\lambda}} \bigg ) \varepsilon_1 = \big ( \frac{4-\tilde{\lambda}}{4} \big )^{\frac{\alpha_1-1}{p-\alpha_1}} < 1 ,
\end{equation*}
the expression $\big \{ 2 + q (|x| - 1)_+ \big \}$ glows at a rate faster than that of $ \frac{4}{4- \tilde{\lambda}} \big \{ 2+ q(\varepsilon_1 |x| -1)_+ + \frac{\tilde{\lambda}}{2} \big \}$, as $|x|$ increases. Hence, it follows that \eqref{keyestimatefinal} holds on $\mathbb{R}^N$ if and only if the following relation holds
\begin{equation}
 \frac{4}{4-\tilde{\lambda}}\bigg \{ 2+ q( \frac{1}{2} -1)_+ + \frac{\tilde{\lambda}}{2} \bigg \} \leq 2 + q \big ( \frac{1}{2\varepsilon_1} -1 \big ) ,
\end{equation}
which, however, is nothing but relation \eqref{check3}. Since the pair $\varepsilon_1 \in (0,1)$, $\alpha_1 \in (1,p)$ definitely satisfies \eqref{check3}, we deduce that $|u_{m+1}(t,x)| \leq 2+ q(|x| -1)_+$ holds on $[-4,0]\times \mathbb{R}^N$. \\
So, inductively, we are able to construct a sequence of numbers $\{d_m\}_{m=1}^{\infty}$ with $|d_m| \leq \frac{\tilde{\lambda}}{2}$ such that for the associated sequence of successively re-scaled functions $\{u_m\}_{m=1}^{\infty}$ as defined in \eqref{recurrencerelation}, we have the following property
\begin{itemize}
\item For each $m \geq 2$, $u_m$ is a weak solution to \eqref{rescaledONE} on $[-4,0]\times B(1)$, and simultaneously a viscosity solution to \eqref{rescaledTWO} on $[-4,0]\times \mathbb{R}^N$. Moreover, the relation $|u_m(t,x)| \leq 2 + q(|x| -1)_+ $ holds on $[-4,0]\times \mathbb{R}^N$ ,
\end{itemize}
which allows us to use either propositions \ref{prop:above} or \ref{prop:below} to obtain the following improved oscillations at finer and finer scales
\begin{equation}\label{improvedOsc}
\big | u_m - d_m \big | \leq 2 - \frac{\tilde{\lambda}}{2} \qquad \mathrm{on}\ \ \ [-1,0]\times B(\frac{1}{2}) ,
\end{equation}
which holds for every $m \in \mathbb{Z}^+$. Then, it follows directly from \eqref{improvedOsc} that the following property holds for all $m \geq 0$
\begin{equation*}
\osc_{Q_m} u_1 \leq 4  \bigg ( \frac{4-\tilde{\lambda}}{4} \bigg )^{m+1} ,
\end{equation*}
in which $Q_m = [-\varepsilon_1^{m\alpha_1} ,0]\times B(\frac{\varepsilon_1^m}{2})$. So, we deduce that $u(0,\cdot ) $ must be Holder continuous at the origin $O \in \mathbb{R}^N$. Actually, for any $x_0 \in \mathbb{R}^N$, we can apply the same argument to the shifted function $u(\cdot , \cdot + x_0)$ in exactly the same way to deduce that the following improved oscillations of $u_1$ holds at smaller and smaller scales around the base point $x_0$.
\begin{equation*}
\osc_{Q_m(x_0)} u_1 \leq 4  \bigg ( \frac{4-\tilde{\lambda}}{4} \bigg )^{m+1} ,
\end{equation*}
where $Q_m(x_0)$ stands for $Q_m(x_0) = [-\varepsilon_1^{m\alpha_1} ,0]\times B_{x_0}(\frac{\varepsilon_1^m}{2})$. This means that the original function $u(0,\cdot )$ is actually Holder's continuous around each $x_0 \in \mathbb{R}^N$. Hence, we may conclude that $u (0,\cdot ) \in C^{\alpha}(\mathbb{R}^N)$ holds for some $\alpha \in (0,1)$, which depends only on $N$, $p$, and $\Lambda$.
\end{proof}
 $$
$$
  \bibliography{ChanVasseur}
\bibliographystyle{plain}

\end{document}